\documentclass[12pt, reqno]{amsart}

\usepackage[initials]{amsrefs}
\usepackage[bookmarks]{hyperref}
\usepackage{graphicx}
\usepackage{amsmath}
\usepackage{bbm}
\usepackage{subfigure}
\usepackage{eepic}
\usepackage{xcolor}
\selectcolormodel{gray}
\usepackage{tikz}
\usepackage{amsmath}
\usepackage{a4wide}
\usepackage[utf8]{inputenc}
\usepackage{amssymb}
\usepackage{amsopn}
\usepackage{epsfig}
\usepackage{amsfonts}
\usepackage{latexsym}
\usepackage{amsthm}
\usepackage{enumerate}
\usepackage[UKenglish]{babel}
\usepackage{verbatim}
\usepackage{color}

\newtheorem{thm}{Theorem}[section]
\newtheorem{lma}[thm]{Lemma}

\newtheorem{defn}[thm]{Definition}

\newtheorem{prop}[thm]{Proposition}
\newtheorem{conj}[thm]{Conjecture}
\newtheorem{rem}[thm]{Remark}

\newcommand{\R}{\mathbb{R}}

\newcommand{\N}{\mathbb{N}}

\providecommand{\norm}[1]{\lVert#1\rVert}

\newcommand{\I}{\mathcal{I}}
\newcommand{\J}{\mathcal{J}}
\newcommand{\A}{\mathcal{A}}
\newcommand{\F}{\mathcal{F}}
\newcommand{\G}{\mathcal{G}}
\renewcommand{\i}{\mathtt{i}}
\renewcommand{\j}{\mathtt{j}}
\renewcommand{\k}{\mathtt{k}}

\renewcommand{\r}{\mathbf{r}}

\newcommand{\hd}{\dim_\textup{H}}
\newcommand{\bd}{\dim_\textup{B}}

\begin{document}
\title[Dimension spectrum of self-affine IFS]{Dimension spectrum of infinite self-affine iterated function systems}
\author{Natalia Jurga}  \address{Mathematical Institute, St Andrews, Scotland, KY16 9SS}
\email{naj1@st-andrews.ac.uk}

\begin{abstract}
  Given an infinite iterated function system (IFS) $\mathcal{F}$, we define its dimension spectrum $D(\mathcal{F})$ to be the set of real numbers which can be realised as the dimension of some subsystem of $\mathcal{F}$. In the case where $\mathcal{F}$ is a conformal IFS, the properties of the dimension spectrum have been studied by several authors. In this paper we investigate for the first time the properties of the dimension spectrum when $\mathcal{F}$ is a non-conformal IFS. In particular, unlike dimension spectra of conformal IFS which are always compact and perfect (by a result of Chousionis, Leykekhman and Urba\'{n}ski, Selecta 2019), we construct examples to show that $D(\mathcal{F})$ need not be compact and may contain isolated points.
\end{abstract}

 \maketitle

 \section{Introduction}

Let $\F=\{S_i: [0,1]^d \to [0,1]^d \}_{ i \in \N}$ be a countable family of Euclidean differentiable contractions whose contraction ratios are uniformly bounded above by some constant $\alpha<1$. We call $\F$ an (infinite) iterated function system (IFS). For each non-empty $\I \subset \N$, one can construct a set $F_\I= \bigcup_{i \in \I} S_i F_\I$, which we call the attractor of $\{S_i:[0,1]^d \to [0,1]^d\}_{ i \in \I}$; see section \ref{sa prel} for its construction. Let $\hd F$ denote the Hausdorff dimension of a set $F \subset \R^d$. The Hausdorff dimension spectrum $D(\F)$ is defined as
$$D(\F):=\{ s \in \R: \hd F_\I=s \; \textnormal{for some} \; \I \subset \N \},$$
that is, the set of dimensions which can be realised by the attractors of subsystems of $\F$. It is easy to see that the dimension spectrum of a finite IFS consists of a finite collection of points. However, when the IFS is taken to consist of infinitely many maps as above, the structure of its dimension spectrum becomes incredibly complex and intriguing. 

Interest in the structure of dimension spectra of infinite iterated function systems began with the so-called \emph{Texan conjecture} concerning the density of the dimensions of bounded type continued fraction sets in $[0,1]$. Recall that for any irrational $x \in (0,1)$ there exists a unique sequence of digits $a_n(x) \in \N$ such that
$$x= \cfrac{1}{a_1(x)+\cfrac{1}{a_2(x)+\cfrac{1}{\cdots}}}.$$
Given $\I \subset \N$ denote
$$E_\I=\{x \in (0,1): a_n(x) \in \I \quad \forall n \in \N\}.$$
Hensley \cite{hensley} and Mauldin and Urba\'{n}ski \cite{mu_conj} independently made the following conjecture, which became known as the Texan conjecture \cite{jenkinson}, due to the fact that they were all at Texan institutions at the time.

\begin{conj}
The set $\{\hd E_\I: \I \subset \N \; \textnormal{finite}\}$
is dense in $[0,1]$.
\end{conj}

Kesseb\"{o}hmer and Zhu \cite{kess} proved the Texan conjecture by showing that the dimension spectrum of the iterated function system $\F=\{S_ix=\frac{1}{x+i}\}_{i \in \N}$ is the closed unit interval $D(\F)=[0,1]$. In particular, 
the density of $\{\hd E_\I: \I \subset \N \; \textnormal{finite}\}$ in $[0,1]$ follows from combining this with the fact that for any $\I \subset \N$, $\hd E_\I= \lim_{n \to \infty} E_{\I \cap \{1, \ldots, n\}}$ and because $E_\I$ is the attractor of the IFS $\{S_i \}_{ i \in \I}$. In addition to proving the Texan conjecture, Kesseb\"{o}hmer and Zhu conducted the first study of dimension spectra of conformal IFS (where all of the maps $S_i$ are conformal), highlighting the relationship between the structure of the spectrum (such as whether it is equal to an interval, a finite union of intervals, a Cantor set etc.) and the decay properties of the sequence of contraction ratios of the maps belonging to the IFS. Since the work of Kesseb\"{o}hmer and Zhu, the dimension spectra of conformal IFS have been further studied in \cite{perfect, rigorous, das_simmons, ghenciu1, ghenciu2, ghenciu_munday}. The conformal IFS in these papers are always assumed to satisfy the \emph{bounded distortion property}: there exists $C>0$ such that for all $n \in \N$, $i_1, \ldots, i_n \in \N$ and all $x,y \in [0,1]^d$
$$C^{-1} \leq \left| \frac{(S_{i_1} \circ \cdots\circ S_{i_n})'(x)}{(S_{i_1} \circ \cdots \circ S_{i_n})'(y)}\right| \leq C.$$
The conformal IFS in these papers are also always assumed to satisfy the open set condition (OSC), which controls the amount of overlap between maps in the IFS. Another (stronger) separation condition we will refer to is the strong open set condition (SOSC), and we will define both conditions in the next section. In the absence of suitable separation conditions, dimension spectra are not yet understood although there are examples to suggest that they have different behaviour, see for instance \cite[p7]{das_simmons}.

In this paper, we will be interested in the topological properties of the dimension spectrum. Chousionis, Leykekhman and Urbański  \cite{perfect} recently proved the following.

\begin{thm}[Theorem 1.2 \cite{perfect}] \label{clu-main}
The dimension spectrum of a conformal IFS satisfying the OSC and bounded distortion property is always compact and perfect.
\end{thm}

In other words, the dimension spectrum of a conformal IFS is always compact with no isolated points. Moreover, Chousionis, Leykekhman and Urbański initiated the investigation into which kinds of compact and perfect sets could be realised as the dimension spectrum of some conformal IFS through asking a number of questions in this direction (see \cite[p4]{perfect}) as well as conjecturing that any compact and perfect subset of $[0, \infty)$ (containing 0) could be realised as the dimension spectrum of some conformal IFS. Das and Simmons \cite{das_simmons} disproved this conjecture by constructing a compact and perfect subset of $[0, \infty)$ containing 0 whose dimension spectrum is not realised by any conformal IFS on $\R^d$. They also constructed an example of a conformal IFS whose dimension spectrum is not uniformly perfect, answering another question from \cite{perfect}.

So far, the study of dimension spectra has been limited to the conformal setting. It is well-known that the dimension theory of non-conformal IFS is considerably more complex than the dimension theory of conformal IFS, and consequently exhibits many new phenomena \cite{gelfert,cp}. Hence, it is natural to ask how this affects the structure of the dimension spectrum. In this paper we investigate for the first time the properties of dimension spectra in the non-conformal setting, by studying the dimension spectra of infinite self-affine IFS. An infinite self-affine IFS is an IFS $\{S_i:[0,1]^d \to [0,1]^d\}_{i \in \N}$ where each map $S_i(\cdot)=A_i(\cdot)+t_i$ for a contraction $A_i \in \mathcal{GL}_d(\R)$ and a translation $t_i \in \R^d$. In particular they provide the simplest models for non-conformal IFS and their dimension theory is a topic of active contemporary research \cite{bhr,bm,das-simmons,fs,morris_shmerkin}.

There are several major obstacles to extending the conformal theory of dimension spectra to the non-conformal setting. Firstly, although \emph{finite} self-affine IFS have received considerable attention over the past few decades, the theory of self-affine \emph{infinite} IFS is still not very well understood, unlike the very well developed theory of conformal infinite IFS \cite{mu_book}. To the author's knowledge, \cite{kaenmaki_reeve} and \cite{reeve} are the only references for this topic. Secondly, in order to study the structure of dimension spectra it is necessary to make good upper and lower estimates on
\begin{eqnarray}
\hd F_\I - \hd F_{\I \setminus \{n\}} \label{dim_change}
\end{eqnarray} for all $\I \subset \N$ and all $n \in \I$. In the conformal setting, this is aided by the multiplicativity of the derivative which appears in the appropriate pressure functions. However, within the appropriate pressure functions that arise in the self-affine setting, the derivative is replaced by a `singular value function' which is in general only submultiplicative but \emph{not} multiplicative, meaning that good upper and lower estimates on (\ref{dim_change}) are difficult to make. Indeed, the question of whether generically (\ref{dim_change}) is strictly positive was only settled in the case $d=3$ recently by Morris and K\"{a}enm\"{a}ki \cite{morris_kaenmaki} and in higher dimensions by Bochi and Morris \cite{bm}. Finally, unlike the conformal setting, the Hausdorff dimension of self-affine sets is not given by a universal formula, and two distinct streams of dimension theory have evolved separately. The Hausdorff dimension of a generic self-affine set is given as the root of a suitable pressure function \cite{hr, bhr, falconer} whereas the Hausdorff dimension of an `exceptional' self-affine set (whenever its value is known) is usually given in terms of a variational principle \cite{baranski, lg}. Since the subsystems of $\F$ can fall into either category, the dimension spectrum of an infinite self-affine IFS cannot simply be described by the roots of a single family of pressure functions, and so in order to understand the global structure of the dimension spectrum of an arbitrary self-affine IFS one must take into consideration both the generic and exceptional theory simultaneously.

% Somehwhere: we always assume sosc, note by an example of das and simmons that strange things can happen if not. / This paper is on top properties, which have been of interest recently / we investigate dimension spectrum in non-conformal case for first time. self-affine sets are known to have more complex dim spectrum and new phenomenon (ref: bhr, ds, more?) so interesting to see how this influenced  spectrum / main obstacles

Let $\mathcal{F}=\{A_i(\cdot)+t_i\}_{i \in \N}$ be a self-affine IFS. We let $\A=\{A_i\}_{i \in \N}$ denote the linear parts of the maps in $\mathcal{F}$. Given $\I \subset \N$, we say that the set of matrices $\{A_i\}_{i \in \I}$ (or simply the self-affine IFS $\{S_i\}_{i \in \I}$) is \emph{irreducible} if the matrices do not all preserve a common proper non-trivial linear subspace. We say that they are \emph{strongly irreducible} if they do not all preserve a common finite union of proper non-trivial linear subspaces. We begin by showing that if we impose some irreducibility assumptions on the subsystems of $\mathcal{F}$, then the dimension spectrum $D(\mathcal{F})$ is compact and perfect, as in the conformal setting.

\begin{thm} \label{cp}
Let $\F$ be a planar self-affine IFS satisfying the SOSC and suppose that any subset of $\A$ is strongly irreducible. Then $D(\F)$ is compact and perfect.
\end{thm}

In contrast, when a self-affine IFS $\mathcal{F}$ is outside of the class that is covered by Theorem \ref{cp}, its dimension spectrum $D(\mathcal{F})$ can have very different topological properties. In particular we will explicitly construct (a) a self-affine IFS whose dimension spectrum is not compact and (b) a self-affine IFS whose dimension spectrum contains isolated points. Recall that by Theorem \ref{clu-main}, neither (a) nor (b) can occur in the conformal setting. 

\begin{thm} \label{main}
There exist infinite self-affine IFS $\F$ satisfying the SOSC such that $D(\F)$ is not compact.
\end{thm}

\begin{thm} \label{isol}
There exist infinite self-affine IFS $\F$ satisfying the SOSC such that $D(\F)$ contains isolated points.
\end{thm}

The conformal analogue of self-affine IFS are self-similar IFS, where each map $S_i$  is a similarity contraction. In the self-similar setting, the structure of the dimension spectrum solely depends on the sequence of contraction ratios of the maps in the self-similar IFS. What Theorems \ref{cp}, \ref{main} and \ref{isol} demonstrate is that in the self-affine case, the structure of the dimension spectrum is not only related to the sequences of singular values of the linear parts of the maps in the self-affine IFS (which describe the contraction ratios in different directions), but also on the irreducibility properties of the set of linear parts. %Maybe compare to CLU?% 

The paper will be organised as follows. In \S 2 we provide some background on self-affine sets and tools for studying their dimension theory. In \S 3 we obtain some estimates on how appropriate `pressure functions'  change when maps from the IFS are added or removed. In \S 4 we use these estimates to prove Theorem \ref{cp}. In \S 5 we prove Theorems \ref{main} and \ref{isol}.

\section{Preliminaries}

\subsection{Finitely and infinitely generated self-affine sets} \label{sa prel}

Given a finite or countable IFS $\{S_i\}_{i \in \I}$ on $\R^d$ with an attractor $F_\I$, we say that the IFS satisfies the \emph{open set condition} (OSC) if there exists an open set $U \subset \R^d$ such that $\bigcup_{i \in \I} S_i U \subset U$ where the union is disjoint. We say that the IFS satisfies the \emph{strong open set condition} if additionally, $F_\I \cap U \neq \emptyset$. 

Let $\{S_i: [0,1]^d \to [0,1]^d \}_{ i \in \I}$ be a finite or countable set of affine contractions; that is, $S_i(\cdot)= A_i (\cdot) +t_i$ with $A_i \in \mathcal{G}\mathcal{L}_d(\R)$ such that $\sup_{i \in \I} \norm{A_i}<1$ where $\norm{\cdot}$ denotes the Euclidean norm, and $t_i \in \R^d$. 

Let $\I^n=\{i_1 \ldots i_n: i_j \in \I\}$ denote words of length $n$ over the index set $\I$, $\I^*=\bigcup_{n \in \N} \I^n$ denote all finite words over the index set and $\Sigma_\I=\I^{\N}$ denote all sequences over the index set. It will also be convenient for us to introduce the notation $\emptyset$ for the `empty word': for $\i \in \I^*$ we let $\emptyset \i$ denote the word $\i$. Given $\i \in \I^*$, let $[\i]$ denote the cylinder set $[\i]=\{\i\j: \j \in \Sigma\}$. Given $\i=i_1 \ldots i_n \in \I^*$ let $|\i|$ denote the length of the word $|\i|=n$. Given $\i=i_1i_2 \ldots \in \Sigma$ or $\i=i_1 \ldots i_{n+m} \in \I^*$ let $\i|_n:= i_1 \ldots i_n$. Given $\i =i_1 \ldots i_n\in \I^*$ let $S_\i:=S_{i_1} \circ \cdots \circ S_{i_n}$ and $A_{\i}:=A_{i_1} \cdots A_{i_n}$. 

Define $\Pi: \Sigma_\I \to \R^d$ as
\begin{eqnarray*}
\Pi(\i) &=& \lim_{n \to \infty} S_{i_1} \circ \cdots \circ S_{i_n}(0) \\ &=& \sum_{k=1}^{\infty} A_{\i|_{k-1}}t_{i_k}.
\end{eqnarray*}
Define $F_\I:=\Pi(\Sigma_\I)$. It is easy to see that $F_\I= \bigcup_{i \in \I} S_i F_\I$, and we call $F_\I$ the attractor of $\{S_i\}_{i \in \I}$. Note that if $\I$ is finite, then $F_\I$ is the unique, non-empty, compact set such that $F_\I= \bigcup_{\i \in \I} S_i F_\I$ by the classical work of Hutchinson \cite{hutchinson}.

\subsection{Singular value function and its multiplicativity properties}

For $A \in \mathcal{G}\mathcal{L}_d(\R)$, let $0<\alpha_d(A) \leq \cdots \leq \alpha_1(A)<1$ denote the $d$ \emph{singular values} of $A$. Define the \emph{singular value function}
\begin{equation*}
\phi^s(A) :=\begin{cases} \alpha_1(A) \cdots \alpha_{\lfloor s \rfloor}(A) \alpha_{\lceil s \rceil}(A)^{s-\lfloor s \rfloor} & s \in [0,d]  \\
|\det A|^{\frac{s}{2}} & s >d \\ \end{cases}. \label{svf}
\end{equation*}
In particular, notice that in the planar setting since $\alpha_1(A)=\norm{A}$ and $\alpha_1(A)\alpha_2(A)=|\det A|$ we have
\begin{equation}
\phi^s(A) =\begin{cases} \norm{A}^{s} & s \in [0,1)  \\
\norm{A}^{2-s} \, |\det(A)|^{s-1} & s \in [1, 2]\\
|\det A|^{\frac{s}{2}} & s >2 \\ \end{cases}. \label{svf2}
\end{equation} 

Note that for any $s \geq 0$, $\phi^s$ is submultiplicative: $\phi^s(AB) \leq \phi^s(A)\phi^s(B)$ for any $A,B \in \mathcal{G}\mathcal{L}_d(\R)$. However, in general $\phi^s$ is not supermultiplicative, that is, it is not generally true that $\phi^s(AB) \geq \phi^s(A)\phi^s(B)$ for $A, B \in \mathcal{GL}_d(\R)$.

We say that $\phi^s$ is \emph{quasimultiplicative} on $\I$ if there exists a constant $c_\I>1$ and a finite subset $\Gamma \subset \I^*$ such that for any $\i, \j \in \I^*$, there exists $\k \in \Gamma$ such that
 $$\frac{1}{c_\I}\phi^s(A_\i)\phi^s(A_\j) \leq  \phi^s(A_{\i\k\j}).$$
We say that the norm is \emph{quasimultiplicative} on $\I$ if $\phi^1(\cdot)=\norm{\cdot}$ is quasimultiplicative on $\I$. In \cite[Proposition 2.8]{feng2}, Feng showed that if $\A=\{A_i\}_{i \in \I}$ is irreducible, then the norm is quasimultiplicative on $\I$ (the theorem was stated for finite sets of matrices but the proof applies verbatim in the infinite case). Note that this implies that $\phi^s$ is quasimultiplicative on $\I$ for any $s \geq 0$, provided $\A \subset \mathcal{GL}_2(\R)$ is irreducible.

\begin{prop}[Proposition 2.8 \cite{feng2}] \label{feng quasi}
If $\{A_i\}_{i \in \I} \subset \mathcal{G}\mathcal{L}_2(\R)$ is irreducible then for all $s \geq 0$, $\phi^s$ is quasimultiplicative on $\I$. 
\end{prop}

Note that \cite[Proposition 2.8]{feng2} ensures quasimultiplicativity for all values of $s$ only in the planar case because in this setting the singular value function can be written as in (\ref{svf2}). In higher dimensions the quasimultiplicativity of the norm implies quasimultiplicativity of $\phi^s$ for $s \in [0,1]\cup[d-1,d]$. On the other hand, in \cite[Lemma 3.5]{morris_kaenmaki} a generalised irreducibility condition was introduced which guarantees quasimultiplicativity of $\phi^s$ for other values of $s$ when $\A \subset \mathcal{GL}_d(\R)$ for arbitary dimension $d$.

A significantly stronger condition is almost multiplicativity. We say that $\phi^s$ is \emph{almost multiplicative} on $\I$ if there exists $c_\I>1$ such that for all $\i, \j \in \I^*$,
$$\frac{1}{c_\I} \phi^s(A_\i)\phi^s(A_\j) \leq \phi^s(A_\i A_\j) \leq \phi^s(A_\i) \phi^s(A_\j).$$

Let $A=\begin{pmatrix}a&b\\c&d \end{pmatrix}$ be a non-negative matrix. Let $\mathbf{1}=\begin{pmatrix}1\\1 \end{pmatrix}$. Note that $\norm{A}':= \mathbf{1}^T A \mathbf{1}=2(a+b+c+d)$ defines a norm on the set of all non-negative matrices. We let $(A)_{(i,j)}$ denote the entry in the $i$th row and $j$th column of $A$.

\begin{defn} \label{kappa}
Suppose $\{A_i\}_{i \in \I}\subset \mathcal{G}\mathcal{L}_2(\R)$ are a set of positive matrices. For any $\mathcal{J} \subset \I^*$ define
$$\kappa(\mathcal{J}):=\min\left\{\frac{(A_\i)_{(i,j)}}{(A_\i)_{(i',j)}} : \i \in \mathcal{J}\right\}.$$
\end{defn}

The constant $\kappa(\mathcal{J})$ describes the projective contraction of the family $\{A_i: i \in \mathcal{J}\}$. In particular, it is determined by $\bigcup_{\i \in \mathcal{J}} A_\i P$, where $P$ denotes the positive cone 
$$P:=\left\{ \begin{pmatrix} x\\y \end{pmatrix} \; : \; x,y >0\right\}.$$ Therefore if $\kappa(\I)>0$ then $\kappa(\I^*)=\kappa(\I)$, since $\bigcup_{\i \in \I^*} A_\i P= \bigcup_{\i \in \I} A_\i P$.

In the following lemma we see that if $\{A_i\}_{i \in \I}\subset \mathcal{G}\mathcal{L}_2(\R)$ are positive matrices with $\kappa(\I)>0$ then $\phi^s$ is almost multiplicative on $\I$. 

\begin{lma} \label{pos almost}
Suppose $\{A_i\}_{i \in \I}\subset \mathcal{G}\mathcal{L}_2(\R)$ is a set of positive matrices with $\kappa(\I)>0$. Denote
$$c := \frac{1}{2} \kappa(\I) \in (0, 1/2).$$
Then for any $i_1, \ldots , i_{n+m} \in \I$,
$$\norm{A_{i_1} \cdots A_{i_{n+m}}}' \geq c \norm{A_{i_1} \cdots A_{i_n}}' \norm{A_{i_{n+1}} \cdots A_{i_{n+m}}}'.$$
\end{lma}

\begin{proof}
We recall the proof from \cite[Lemma 2.1]{feng} (which provides a $d$-dimensional version of this result) for completeness. First observe that since $\kappa(\I^*)=\kappa(\I)$, for any $\i \in \I^*$, $A_\i \geq c E A_\i$ where $E= \begin{pmatrix} 1&1\\1&1 \end{pmatrix}$. 
Therefore,
\begin{eqnarray*}
\norm{A_{i_1} \cdots A_{i_{n+m}}}' &\geq& \norm{A_{i_1} \cdots A_{i_n} c E A_{i_{n+1}} \cdots A_{i_{n+m}}}'\\
&=& c \norm{A_{i_1} \cdots A_{i_n} \mathbf{1} \mathbf{1}^T A_{i_{n+1}} \cdots A_{i_{n+m}}}' \\
&=& c \mathbf{1}^T A_{i_1} \cdots A_{i_n} \mathbf{1} \mathbf{1}^T A_{i_{n+1}} \cdots A_{i_{n+m}} \mathbf{1} \\
&=& c \norm{A_{i_1} \cdots A_{i_n}}' \norm{A_{i_{n+1}} \cdots A_{i_{n+m}}}'.
\end{eqnarray*}
\end{proof}

Recently in \cite{bmk}, B\'{a}r\'{a}ny, Morris and K\"{a}enm\"{a}ki studied more general properties of the semigroup generated by $\A$ which ensure almost multiplicativity of the norm.

\subsection{Pressure and dimension theory}

Define the pressure function $P_\I:[0, \infty) \to \R \cup \{\infty\}$ by
$$P_\I(s):= \lim_{n \to \infty} \left( \sum_{\i \in \I^n} \phi^s(A_\i)\right)^{\frac{1}{n}}.$$
The pressure $P_\I(s)$ is a strictly decreasing function of $s$ and is therefore finite on an interval $(\theta, \infty)$, where $\theta:= \inf \{s \geq 0: P_\I(s)< \infty\}$ is called the \emph{finiteness parameter} for the system $\{A_i\}_{i \in \I}$. In particular $\theta=0$ if $\I$ is finite, but can be strictly positive if $\I$ is infinite. $P_\I$ is a continuous and convex function of $s$ on $(\theta, \infty)$.

We define the \emph{affinity dimension} to be 
$$s(\I):= \inf\{s \geq 0: P_\I(s) \leq 1\}.$$
Note that since all matrix norms are equivalent, $P_\I(s)$ and $s(\I)$ are actually independent of the matrix norm used in the definition of $\phi^s$. 

In his seminal paper \cite{falconer}, Falconer introduced the objects defined above and showed that if $\I$ is finite then $\min\{s(\I),d\}$ is an upper bound on $\hd F_\I$, and is equal to $\hd F_\I$ for generic choices of translations. Since then the focus of research on dimension theory of finitely generated self-affine sets has been to determine the dimension of explicit families of self-affine sets, which has split into two independent strands of research: one in which one strives to characterise families of self-affine sets which obey the generic rule \cite{bhr,hr} and one in which one strives to obtain the dimension of `exceptional' self-affine sets for which the generic formula does not hold \cite{lg,baranski}. The following theorem of B\'{a}r\'{a}ny, Hochman and Rapaport \cite{bhr} falls into the first category, where it was shown  that in the planar setting, $\hd F_\I= s(\I)$ outside of a small family of exceptions.

\begin{thm}[Theorem 1.1 \cite{bhr}]
Suppose $\I$ is finite and $\{S_i: i \in \I\}$ is a strongly irreducible planar self-affine IFS which satisfies the SOSC. Then $\hd F_\I= s(\I)$. \label{bhr}
\end{thm}

%Absence of separation conditions such as the SOSC is not the only mechanism which yields dimension drop. For example, whenever the matrices $\{A_i\}_{i \in \I}$ are not strongly irreducible the self-affine set may belong to the small class of exceptions for which the Hausdorff dimension drops from the `expected' value $s(\I)$. This includes the Feng-Wang carpets studied in \cite{fw}. We let $\bd F$ denote the box dimension of a set $F$.

%\begin{thm}[Theorem 2.2 \cite{fw}]
%Suppose $\I$ is finite and $A_i= \begin{pmatrix} a_i&0\\0&b_i \end{pmatrix}$ for all $i \in \I$, where $|a_i| \geq |b_i|$. Suppose that $\{S_i([0,1]^2)\}_{i \in \I}$ are pairwise disjoint. Let $\pi$ denote projection to the $x$ axis. Then $\bd F_\I=s$ where $s$ is defined implicitly by
%$$\sum_{i \in \I} a_i^{\bd \pi F_\I} b_i^{s-\bd \pi F_\I}=1.$$ \label{fw}
%In particular, $\hd F_\I \leq \bd F_\I < s(\I)\leq 2$ whenever $\bd \pi F_\I<1$.
%\end{thm}

The dimension theory of infinite self-affine IFS is less understood, however when the norm is quasimultiplicative one can recover many analogues of theorems concerning finite self-affine IFS. In \cite{kaenmaki_reeve}, K\"{a}enm\"{a}ki and Reeve proved the following.

\begin{thm}[\cite{kaenmaki_reeve}]
Suppose $\I \subset \N$ is infinite and $\{A_i\}_{i \in \I}\subset \mathcal{GL}_2(\R)$ are irreducible. Suppose there exists a sequence of finite sets $\I_1 \subset \I_2 \subset \ldots$ such that $\I= \bigcup_{n \in \N} \I_n$ where $\hd F_{\I_n}=s(\I_n)$ for all $ n \in \N$. Then
\begin{equation}
\hd F_\I =s(\I). \label{kr eq}
\end{equation} \label{kr}
\end{thm}

\begin{proof} The result can be gleaned from the work in \cite{kaenmaki_reeve}, but since it is not stated in this exact form, we detail the relevant results from their paper. Observe that by Proposition \ref{feng quasi}, if $n$ taken sufficiently large that $\Gamma \subset \I_n$, then $\phi^s$ is quasi-multiplicative on $\I_n$, and the constant from the definition of quasi-multiplicativity can be taken uniformly. Therefore by \cite[Proposition 3.2]{kaenmaki_reeve}, $P_\I(s)= \lim_{n \to \infty} P_{\I_n}(s)$. In particular, as noted in the proof of \cite[ Theorem B]{kaenmaki_reeve} (on the bottom of page 12), 
$$s(\I)=\lim_{n \to \infty} s(\I_n)= \lim_{n \to \infty} \hd F_{\I_n},$$
where the last equality follows by our assumption. We also see from the argument on the top of page 13 in \cite[ Theorem B]{kaenmaki_reeve} that since $s(\I_n)= \hd F_{\I_n}$ for all $n \in \N$, then $\hd F_\I=s(\I)$.
\end{proof}

Note that if $\{A_i\}_{i \in \I}\subset \mathcal{GL}_2(\R)$ is strongly irreducible then directly from Theorem \ref{bhr} and Theorem \ref{kr} one obtains $\hd F_\I=s(\I)$.

\section{Pressure estimates}

In order to analyse the dimension spectrum, it is necessary to obtain lower and upper bounds on $|P_{\I \cup \J}(s)-P_\I(s)|$ when $\J$ is either a digit or a set of digits from $\N \setminus \I$. When the IFS $\{S_i\}_{i \in \I}$ is conformal, the dimension is given by the root of the pressure function $\lim_{n\to \infty} \left(\sum_{\i \in \I^n} \norm{S'_\i}_{\infty}^s\right)^{\frac{1}{n}}$, and the multiplicativity of the derivative (i.e. the chain rule for $S_\i'$) makes estimating $|P_{\I \cup \J}(s)-P_\I(s)|$ relatively straightforward. For example, in the self-similar setting, we can explicitly compute $|P_{\I \cup \J}(s)-P_\I(s)|=\sum_{i \in \J} r_i^s$, and the more general conformal setting is not much more difficult as long one has a bounded distortion property.

The non-multiplicativity of $\phi^s$ makes the problem of obtaining bounds on $|P_{\I \cup \J}(s)-P_\I(s)|$ more challenging in the nonconformal setting, and is related to the recently settled folklore open problem of whether removing a map from a self-affine iterated function system results in a strict drop in the affinity dimension. By using the variational principle for $P_{\I}$ and $P_{\I \cup \J}$ and by proving that equilibrium states of $\phi^s$ are fully supported,  which was established in dimension $d=2$ by Feng and K\"{a}enm\"{a}ki \cite{fk}, in dimension $d=3$ by K\"{a}enm\"{a}ki and Morris \cite{morris_kaenmaki} and in arbitrary dimension by Bochi and Morris \cite{bm}, it has recently been proved that $P_{\I \cup \J}(s)>P_{\I}(s)$ (hence $s(\I \cup \J)> s(\I)$). Unfortunately for us, this method of proof yields no immediate bounds on \emph{how much} the pressure $P_{\I \cup \J}(s)$ increases from $P_\I(s)$. 

In order to prove Theorem \ref{cp}, we only need to compute an upper bound on $|P_{\I \cup \J}(s)-P_\I(s)|$ in the case that $\phi^s$ is quasimultiplicative on $\I$, which can be found in Lemma \ref{bounded change2}. On the other hand, in the almost multiplicative setting we obtain lower and upper bounds on $|P_{\I \cup \J}(s)-P_\I(s)|$, see Lemma \ref{bounded change}, which will be necessary to construct the example of a dimension spectrum which is not compact.

Before we state and prove these lemmas, we introduce some notation. Let $\I \subset \N$, $\J \subset \N \setminus \I$ and $n \in \N$. For $0 \leq j \leq n$, let $\Lambda_j$ denote all words in $(\I \cup \J)^n$ that contain $j$ instances of digits from the set $\J$. Define the \emph{ordering} of $\i=i_1 \ldots i_n \in (\I \cup \J)^n$ to be the $n$-tuple of 0s and 1s $(x_1, \ldots, x_n)$ where $x_t=1$ if $i_t \in \I$ and $x_t=0$ if $i_t \in \J$. Notice that there are ${n \choose j}$ possible orderings of words in $\Lambda_j$ and so we enumerate each of these by $1 \leq \tau \leq {n \choose j}$. We denote the set of all $\i \in \Lambda_j$ with ordering $\tau$ by $\Lambda_{j,\tau}$. We let $k_1, \ldots, k_{l_\tau}>0$ denote the lengths of consecutive strings of 1s in the $\tau$th ordering in $\Lambda_j$ so that $\sum_{t=1}^{l_\tau} k_t=n-j$ and $l_\tau \leq j+1$. Similarly, we let $m_1, \ldots, m_{\ell_\tau}$ denote the lengths of consecutive strings of 0s in the $\tau$th ordering in $\Lambda_j$ so that $\sum_{t=1}^{\ell_\tau} m_t=j$ and $\ell_\tau \leq j$.

\begin{lma} \label{bounded change}
Let $\I \subset \N$ and $\J \subset \N \setminus \I$. Suppose $\phi^s$ is almost multiplicative on $\I$ and $\I \cup \J$, for constants $c_\I, c_{\I \cup \J}>1$ respectively. Then for any $s>0$,
\begin{eqnarray}
\label{eq bc} P_{\I}(s) + \frac{1}{c_{\I \cup \J}^2}P_\J(s) \leq P_{\I \cup \J}(s) \leq P_{\I }(s) + c_\I\sum_{i \in \J} \phi^s(A_i). \end{eqnarray}
\end{lma}

\begin{proof}
Fix $n \in \N$ and $0 \leq j\leq n$. We have
$$\sum_{\i \in (\I\cup\J)^n} \phi^s(A_\i)=\sum_{j=0}^n \sum_{\tau=1}^{{n \choose j}} \sum_{\i \in \Lambda_{j,\tau}} \phi^s(A_\i)$$
and
\begin{eqnarray*}
\sum_{\i \in \Lambda_{j,\tau}} \phi^s(A_\i) &\leq& \left(\sum_{i \in \J} \phi^s(A_i)\right)^j  \sum_{\i_1 \in \I^{k_1}} \cdots \sum_{\i_{l_\tau} \in \I^{k_{l_\tau}}} \phi^s(A_{\i_1}) \cdots \phi^s(A_{\i_{l_\tau}}) \\
&\leq& \left(\sum_{i \in \J} \phi^s(A_i)\right)^j c_{\I}^{l_{\tau}-1} \sum_{\i_1 \in \I^{k_1}} \cdots \sum_{\i_{l_\tau} \in \I^{k_{l_\tau}}} \phi^s(A_{\i_1 \ldots \i_{l_\tau}}) \\
&\leq& \left(\sum_{i \in \J} \phi^s(A_i)\right)^j c_{\I}^{j} \sum_{\i \in \I^{n-j}} \phi^s(A_\i).
\end{eqnarray*}
Also,
\begin{eqnarray*}
\sum_{\i \in \Lambda_{j,\tau}} \phi^s(A_\i)  &\geq& c_{\I \cup \J}^{-(l_\tau +\ell_\tau-1)} \sum_{\j_1 \in \J^{m_1}} \cdots \sum_{\j_{\ell_\tau} \in \J^{m_{\ell_{\tau}}}} \phi^s(A_{\j_1}) \cdots \phi^s(A_{\j_{\ell_\tau}})\sum_{\i_1 \in \I^{k_1}} \cdots \sum_{\i_{l_\tau} \in \I^{k_{l_\tau}}} \phi^s(A_{\i_1}) \cdots \phi^s(A_{\i_{l_\tau}}) \\
&\geq&  c_{\I \cup \J}^{-2j} \sum_{\j \in \J^j} \phi^s(A_\j) \sum_{\i \in \I^{n-j}} \phi^s(A_\i).
\end{eqnarray*}

Note that if $\phi^s$ is almost multiplicative on $\I \subset \N$ with constant $c_{\I}>1$ then for any $n \in \N$, 
$$P_\I(s)^{n} \leq \sum_{\i \in \I^{n}}\phi^s(A_\i) \leq c_\I P_\I(s)^{n}.$$

Therefore
\begin{eqnarray*}
P_{\I \cup \J}(s)&\leq& \lim_{n \to \infty} \left( \sum_{j=0}^n \sum_{\tau=1}^{{n \choose j}} \left(\sum_{i \in \J} \phi^s(A_i)\right)^j c_{\I}^{j} \sum_{\i \in \I^{n-j}} \phi^s(A_\i) \right)^{\frac{1}{n}} \\
&\leq& \lim_{n \to \infty} \left( \sum_{j=0}^n {n \choose j} \left(\sum_{i \in \J} \phi^s(A_i)\right)^j c_{\I}^{j+1} (P_{\I }(s))^{n-j} \right)^{\frac{1}{n}}\\
&=& c_\I\sum_{i \in \J} \phi^s(A_i)+ P_{\I }(s).
\end{eqnarray*}
Similarly, we also have
\begin{eqnarray*}
P_{\I \cup \J}(s)&\geq& \lim_{n \to \infty} \left( \sum_{j=0}^n \sum_{\tau=1}^{{n \choose j}} c_{\I \cup \J}^{-2j} \sum_{\j \in \J^j} \phi^s(A_\j) \sum_{\i \in \I^{n-j}} \phi^s(A_\i)  \right)^{\frac{1}{n}} \\
&\geq& \lim_{n \to \infty} \left( \sum_{j=0}^n {n \choose j} c_{\I \cup \J}^{-2j}P_\J(s)^j P_\I(s)^{n-j} \right)^{\frac{1}{n}}\\
&=& \frac{1}{c_{\I \cup \J}^2}P_\J(s)+ P_{\I }(s).
\end{eqnarray*}
\end{proof}

Notice that if there exists a constant $C>0$ for which $C^{-1} \leq \frac{\tilde{\phi}^s(A_\i)}{\phi^s(A_\i)} \leq C$ for all $\i \in \N^{\N}$ and $\tilde{\phi}^s$ is almost multiplicative on $\I$ and $\I\cup \J$ for constants $\tilde{c}_\I, \tilde{c}_{\I \cup \J}>1$ respectively, then the analogue of (\ref{eq bc}) also holds where $c_\I$, $c_{\I \cup \J}$ and $\phi^s$ are replaced by $\tilde{c}_\I, \tilde{c}_{\I \cup \J}$ and $\tilde{\phi}^s$ respectively. Using this observation we rephrase Lemma \ref{bounded change} for positive matrices in $\mathcal{GL}_2(\R)$ explicitly in terms of the matrix entries. Recall that for $\I \subset \N$, $\kappa(\I)$ was defined in Definition \ref{kappa} as the minimum possible column ratio of matrices indexed by a digit in $\I$. We also recall that for a positive matrix $A \in \mathcal{GL}_2(\R)$, $\norm{A}'$ simply denotes the sum of its entries.

\begin{lma} \label{positive change}
Let $\I \subset \N$ and $\J \subset \N \setminus \I$. Suppose that $\kappa(\I), \kappa(\J)>0$. Then for any $0<s\leq 1$,
\begin{eqnarray}
\label{eq pc} P_{\I}(s) + \frac{(\kappa(\I \cup \J))^{2s}}{4^s}P_\J(s) \leq P_{\I \cup \J}(s) \leq P_{\I }(s) + \left(\frac{2}{\kappa(\I)}\right)^s\sum_{i \in \J} (\norm{A_i}')^s \end{eqnarray}
and for any $1 <s \leq 2$,
\begin{eqnarray}
\label{eq pc2} P_{\I}(s) + \frac{(\kappa(\I \cup \J))^{4-2s}}{4^{2-s}}P_\J(s) \leq P_{\I \cup \J}(s) \leq P_{\I }(s) + \left(\frac{2}{\kappa(\I)}\right)^{2-s}\sum_{i \in \J} (\norm{A_i}')^{2-s} |\det A_i|^{s-1}. \end{eqnarray}
\end{lma}

\begin{proof}
Proof follows directly from Lemma \ref{pos almost} and the fact that all matrix norms are equivalent. 
\end{proof}

Next we tackle the quasimultiplicative setting.

\begin{lma} \label{bounded change2}
Suppose $\phi^s$ is quasimultiplicative on $\I$ with constant $c_\I>1$. For each $s >0$ there exists a constant $C_\I$ that depends only on $\I$ and $s$ such that for any $\J \subset \N \setminus \I$,
$$P_\I(s)< P_{\I \cup \J}(s) \leq P_\I(s) + C_\I \sum_{i \in \J} \phi^s(A_i).$$
\end{lma}

\begin{proof}
The lower bound follows from the strict monotonicity of the affinity dimension \cite[Theorem 2]{bm}, therefore it is sufficient to prove the upper bound. 

As before we fix $n \in \N$ and $0 \leq j \leq n$ and write
$$\sum_{\i \in (\I\cup\J)^n} \phi^s(A_\i)=\sum_{j=0}^n \sum_{\tau=1}^{{n \choose j}} \sum_{\i \in \Lambda_{j,\tau}} \phi^s(A_\i).$$
In particular,
$$\sum_{\i \in \Lambda_{j,\tau}} \phi^s(A_\i) \leq \left(\sum_{i \in \J} \phi^s(A_i)\right)^j  \sum_{\i_1 \in \I^{k_1}} \cdots \sum_{\i_{l_\tau} \in \I^{k_{l_\tau}}} \phi^s(A_{\i_1}) \cdots \phi^s(A_{\i_{l_\tau}}) .$$
By \cite[Lemma 3.1]{kaenmaki_reeve}, since $\phi^s$ is quasimultiplicative on $\I$,
$$\sum_{\i \in \I^n} \phi^s(A_\i) \leq c_\I K \max\{1,P_\I(s)^K\}P_\I(s)^n$$
where $K= \max\{|\i|: \i \in \Gamma\}$. Put $C_\I=c_\I K \max\{1,P_\I(s)^K\}$, which clearly depends only on $\I$ and $s$. Then
\begin{eqnarray*}
\sum_{\i \in \Gamma_{\j, \tau}} \phi^s(A_\i) &\leq& \left(\sum_{i \in \J} \phi^s(A_i)\right)^j  C_\I^{l_\tau} P_{\I}(s)^{k_1 + \cdots +k_{l_\tau}}\\
&\leq &\left(\sum_{i \in \J} \phi^s(A_i)\right)^j  C_\I^{j}  P_{\I}(s)^{n-j}.
\end{eqnarray*}
 Therefore,
\begin{eqnarray*}
P_{\I}(s)=\lim_{n \to \infty}\left(\sum_{\i \in \I^n} \phi^s(A_\i)\right)^{\frac{1}{n}}&\leq& \lim_{n \to \infty}\left(\sum_{j=0}^n \sum_{\tau=1}^{{n \choose j}} \left(\sum_{i \in \J} \phi^s(A_i)\right)^j  C_\I^{j}  P_{\I}(s)^{n-j}\right)^{\frac{1}{n}} \\
&=&\lim_{n \to \infty}\left(\sum_{j=0}^n {n \choose j} \left(\sum_{i \in \J} \phi^s(A_i)\right)^j  C_\I^{j}  P_{\I}(s)^{n-j}\right)^{\frac{1}{n}} \\
&=& C_\I\left(\sum_{i \in \J} \phi^s(A_i)\right) + P_{\I}(s),
\end{eqnarray*}
proving the upper bound.

\end{proof}

\section{Compactness and perfectness of the affinity spectrum}

Throughout this section we  consider $\A=\{A_i\}_{i \in \N} \subset \mathcal{GL}_2(\R)$ with the property that any subset of $\A$ is irreducible. Hence by Proposition \ref{feng quasi}, for all $s \geq 0$, $\phi^s$ is quasimultiplicative on any $\I \subset \N$. We define the affinity spectrum $D(\A)$ of $\A$ as the analogue of the Hausdorff dimension spectrum for the affinity dimension, that is,
$$D(\A):=\{ s(\I) \; : \; \I \subset \N\}.$$
We will prove that under the above irreducibility assumption, $D(\A)$ is compact and perfect.  Theorem \ref{cp} will be a straightforward corollary of this. 

 Recall that the finiteness parameter $\theta$ is defined as the unique real number such that $P_\N(s)= \infty$ for $s<\theta$ and $P_\N(s)< \infty$ for $s> \theta$. Note that $P_\N(\theta)$ can either be finite or infinite. Now that we are equipped with the pressure estimates from the previous section, $D(\A)$ can be shown to be compact in a similar way to how the dimension spectrum of conformal IFS was shown to be compact in \cite{perfect}.

\begin{lma} \label{compact}
Suppose every subset of $\A\subset \mathcal{GL}_2(\R)$ is irreducible. Then $D(\A)$ is compact.
\end{lma}

\begin{proof}
Let $\{\I_n\}$ be a sequence of subsets of $\N$ such that $\lim_{n \to \infty} s(\I_n)=s$, equivalently $\lim_{n \to \infty} P_{\I_n}(s)=1$. We will construct a set $\I \subset \N$ such that $P_\I(s)=1$. First we consider the case that $P_\N(s)< \infty$, that is, either $s> \theta$ or $s=\theta$ and $P_\N(\theta)< \infty$.

First, suppose there does not exist $k \in \N$ such that $k \in \I_n$ for infinitely many $n \in \N$. Note that for sufficiently large $N \in \N$, $P_{\{N, N+1, \ldots\}}(s) \leq  \sum_{n=N}^{\infty} \phi^s(A_n)<1$ and therefore $s(\{N, N+1, \ldots\})<s$. Therefore, for all $n$ sufficiently large, $\I_n$ must intersect the set $\{1, \ldots, N-1\}$. In particular there must be at least one $k \in \{1, \ldots, N-1\}$ such that $k \in \I_n$ for infinitely many $n \in \N$. We let $k_1$ denote the smallest digit with this property:
$$k_1:=\min\{k \in \N \; : \; k \in \I_n \; \textnormal{for infinitely many} \; n \in \N\}.$$
Also, let
$$G_1:= \{n \in \N \; : \; k_1 \in \I_n\}$$
and
$$G_1^k:= \{n \in G_1 \; : \; k \in \I_n\}.$$
We claim that for some $k \in \N \setminus \{k_1\}$, $G_1^k$ is an infinite set. To see this, observe that $P_{\{k_1\}}(s) \leq \phi^s(A_{k_1})<1$, therefore for sufficiently large $N \in \N$,
$$P_{\{k_1, N, N+1, \ldots\}}(s) \leq \phi^s(A_{k_1}) +  \sum_{n=N}^{\infty} \phi^s(A_n) < 1.$$
Therefore, for sufficiently large $n \in G_1$, each $\I_n$ must intersect $\{1, \ldots, N-1\} \setminus \{k_1\}$, and therefore there exists $k \in \{1, \ldots, N-1\} \setminus \{k_1\}$ such that $k \in \I_n$ for infinitely many $ n \in G_1$. Let $k_2$ denote the smallest digit with this property, and observe that $k_2>k_1$.

Now, define $G_2=G_1^{k_2}$. Inductively we can define a collection $\{k_i\}_{i=1}^p$ where $2 \leq p \leq \infty$ and $G_1 \supset G_2 \ldots$ such that
$$G_{i+1}:=\{n \in G_i \; : \; k_{i+1} \in \I_n\}$$
and
$$k_{i+1}:= \min\{k \in \N \; : \; k \in \I_n \setminus \{k_1, \ldots, k_i\} \; \textnormal{for infinitely many } \; n \in G_i\}.$$
Now, put $\I=\{k_i\}_{i=1}^p$. Notice that by construction, for any $N \in \N$, $\I \cap \{1, \ldots, N\} \subset \I_n$ for infinitely many $n \in \N$. Therefore
$$s(\I \cap \{1, \ldots, N\}) \leq \lim_{n \to \infty} s(\I_n)=s$$
and in particular $s(\I) \leq s$, since by Theorem \ref{kr}, $s(\I \cap \{1, \ldots, N\}) \to s(\I)$ as $N \to \infty$.

Now, suppose for a contradiction that $s(\I)<s$. By assumption $\{A_i\}_{i \in \I}$ is irreducible, so we can let $C_\I$ denote the constant from Lemma \ref{bounded change2}. For sufficiently large $N \in \N$,
$$P_{\I \cup \{N, N+1, \ldots\}}(s)< P_\I(s)+ C_\I \sum_{n=N}^{\infty} \phi^s(A_n)<1$$
and therefore for all sufficiently large $n \in \N$, $\I_n$ must intersect $\{1, \ldots, N-1\} \setminus \I$. If $p$ is finite this immediately implies there exists $k \in \{1, \ldots, N-1\} \setminus \{k_i\}_{i=1}^p$ such that $k \in \I_n$ for infinitely many $n \in G_p$, contradicting the fact that $k_p$ is the last digit with this property. On the other hand if $p = \infty$ then we can fix $q$ sufficiently large such that $k_q>N$. Then there exists $k \in \{1, \ldots, N-1\} \setminus \{k_i\}_{i=1}^p$ such that for infinitely many $n \in G_q$, $k \in \I_n $. This contradicts that $k_{q}$ was chosen minimally. Since either way we obtain a contradiction, it follows that $s(\I)=s$. 

Next we tackle the case where $P_\N(s)= \infty$, that is, either $s< \theta$ or $s= \theta$ and $P_\N(\theta)= \infty$. We will prove that $s$ is in the spectrum, thus proving that $[0,\theta]$ is contained in the spectrum, i.e. that part of the spectrum is compact. Note that since $P_\N(s)= \infty$, this implies that $P_{\{N, N+1, \ldots\}}(s)= \infty$ for any $N \in \N$. To see this, fix $N$ and note that since $\{A_i\}_{i \in \{N, N+1, \ldots\}}$ is irreducible, we can denote the constant from Lemma \ref{bounded change2} by $C_{\{N, N+1, \ldots\}}$. Observe that
$$\infty= P_\N(s) \leq P_{\{N, N+1, \ldots\}}(s)+ C_{\{N, N+1, \ldots\}} \sum_{n=1}^{N-1} \phi^s(A_n).$$
Since the second term is finite, this implies $P_{\{N, N+1, \ldots\}}(s)= \infty$. 

Without loss of generality we can assume that there does not exist a finite set $\I \subset \N$ such that $P_\I(s)=1$. Define $\I_1=\{1, \ldots, N\}$ where $N$ is chosen such that $P_{\I_1}(s)<1$ and $P_{\I_1 \cup \{\max \I_1 +1\}}(s)>1$. Define a sequence of sets $\I_n$ inductively by $\I_n= \I_{n-1} \cup \{\max \I_{n-1} +N_1, \ldots, \max \I_n+ N_2\}$ where $N_1$ is chosen such that $P_{\I_{n-1} \cup \{\max \I_{n-1} +N_1-1\}}(s)>1$ whereas $P_{\I_{n-1} \cup \{\max \I_{n-1}+N_1\}}(s)<1$ and $N_2$ is chosen such that $P_{\I_{n}}(s)<1$ but $P_{\I_{n} \cup \{\max \I_{n}+1\}}(s)>1$. Note that the construction of the sets $\I_n$ is possible since for any $N \in \N$, $P_{\{N, N+1, \ldots\}}(s)=\infty$. Let $\I_n'=\I_n \cup \{\max \I_n +1\}$ and $\I= \bigcup_{n \in \N} \I_n$. Let $\Gamma$ be the set of connectors for $\I$ and notice that for all $n$ sufficiently large, $\Gamma \subset \I_n$, hence $C_{\I_n}=C_\I$. In what follows we consider $n$ sufficiently large for this to be the case.

Since $P_{\I_n}(s)<1<P_{\I_n'}(s)$ for all $n \in \N$, and $P_{\I_n}(s)-P_{\I_n'}(s) \leq C_{\I_n} \phi^s(A_{\max \I_n+1}) = C_\I \phi^s(A_{\max \I_n+1}) \to 0$ as $n \to \infty$, it follows that $P_{\I_n}(s) \to 1$ as $n \to \infty$. Also, for any $N \in \N$ and $n \in \N$, $P_{\I \cap \{1, \ldots, N\}}(s)<1$ and therefore it follows by Theorem \ref{kr} that $P_{\I}(s) \leq 1$. Moreover, $1-P_{\I}(s)\leq P_{\I_n'}(s)-P_{\I_n}(s) \to 0$ as $n \to \infty$ and therefore $P_{\I}(s)= \lim_{n \to \infty} P_{\I_n}(s)=1$, thus we are done.  
\end{proof}

We can also use the pressure estimates from the previous section to prove perfectness of $D(\A)$.

\begin{lma} \label{perfect}
Suppose every subset of $\A\subset \mathcal{GL}_2(\R)$ is irreducible. Fix $\I \subset \N$ and $\epsilon>0$. There exists $\J \subset \N$, $\J \neq \I$ such that $0<|s(\I)-s(\J)|< \epsilon$. In particular $D(\A)$ is perfect.
\end{lma}

\begin{proof}
Fix $\I \subset \N$. First we consider the case that $\I$ is infinite. By Theorem \ref{kr}, $\lim_{n \to \infty} s(\I \cap \{1,\ldots,n\})=s(\I)$. Moreover, the monotonicity of the affinity dimension \cite{bm} implies that for any $n \in \N$, $s(\I \cap \{1,\ldots,n\})<s(\I)$, completing the proof. 

Next, if $\I$ is finite, let $s>s(\I)$. By Lemma \ref{bounded change2} we have
$$P_\I(s)<P_{\I \cup \{n\}}(s) \leq P_\I(s)+C_\I \phi^s(n)$$
for any $n \in \N \setminus \I$. Since $P_\I(s)<1$ we can choose $n \in \N \setminus \I$ sufficiently large that $P_{\I \cup \{n\}}(s)<1$. In particular, $s<s(\I \cup\{n\})<s(\I)$, completing the proof.
\end{proof}

Note that the proofs of Lemmas \ref{compact} and \ref{perfect} both follow from quasimultiplicativity of $\phi^s$ for all $s \in [0,2]$. As such, we can also assert that $D(\A)$ is compact and perfect when $\A \subset \mathcal{GL}_d(\R)$ satisfies the irreducibility property that guarantees quasimultiplicativity of $\phi^s$ for all $s \in [0,d]$, see \cite[Lemma 3.5]{morris_kaenmaki}.
 
\vspace{2mm}

\noindent \emph{Proof of Theorem \ref{cp}.} If every subset of $\A$ is strongly irreducible and if the translations in $\F=\{S_i(x)=A_i(x)+t_i\}_{i \in \I}$ are chosen such that $\F$ satisfies the SOSC,  then $D(\F)=D(\A)$ by  Theorem \ref{bhr}. Thus, the proof follows from Lemmas \ref{perfect} and \ref{compact}. \qed

\section{Non-compact dimension spectra containing isolated points}

In this section we investigate the dimension spectrum of self-affine systems which lie outside the scope of Theorem \ref{cp}. In particular, in Section \ref{original} we prove Theorem \ref{main} by constructing a self-affine IFS whose dimension spectrum is not compact and in Section \ref{annoying} we prove Theorem \ref{isol} by constructing a self-affine IFS whose dimension spectrum contains an isolated point.

 % Firstly, in section \ref{original} we construct a self-affine system $\mathcal{F}$ with linear parts $\A$ which contains a finite reducible subsystem. In particular, $D(\mathcal{F})$ will coincide with $D(\A)$, except possibly at a finite set of points (those points associated to the reducible subsystems). We will choose the translations corresponding to the reducible subsystem in such a way that its dimension will `drop' into a hole of the self-affine spectrum, producing an isolated point, and separately causing $D(\mathcal{F})$ to not be closed. Therefore, in this scenario the loss of nice topological properties is caused by points in the dimension spectrum associated to reducible subsystems disrupting the compact and perfect structure inherited from the affinity spectrum. Secondly, in section \ref{annoying} we construct a reducible self-affine system whose dimension spectrum is not compact and contains isolated points.  

\subsection{Example of a non-compact dimension spectrum} \label{original}

In this section we construct a self-affine IFS whose dimension spectrum is not closed, hence not compact, proving Theorem \ref{main}. 

Let $3<\beta<\gamma$ and $\frac{1}{2}<b<d<1$. Fix $A^{\beta,\gamma}=\begin{pmatrix}\frac{1}{\beta}&0\\0&\frac{1}{\gamma}\end{pmatrix} $ and
$$A^{\beta,\gamma}_1=A^{\beta,\gamma}_2=A^{\beta,\gamma}_3= \begin{pmatrix} 1&-1 \\ \frac{1}{2}&\frac{1}{2}\end{pmatrix} \begin{pmatrix}\frac{1}{\beta}&0\\0&\frac{1}{\gamma}\end{pmatrix} \begin{pmatrix} \frac{1}{2}&1 \\ -\frac{1}{2}&1\end{pmatrix} = \begin{pmatrix} \frac{1}{2}(\frac{1}{\beta}+\frac{1}{\gamma}) & \frac{1}{\beta}-\frac{1}{\gamma} \\ \frac{1}{4}(\frac{1}{\beta}-\frac{1}{\gamma}) & \frac{1}{2}(\frac{1}{\beta}+\frac{1}{\gamma}) \end{pmatrix}$$
noting that this is a conjugation of $A^{\beta,\gamma}$. Consider the set of matrices $\A_{\beta,\gamma}=\{A^{\beta,\gamma}_n\}_{n \in \N \setminus \{4\}}$ where for $n \geq 5$,
$$A^{\beta, \gamma}_n= \begin{pmatrix} \frac{1}{\beta^n} & \frac{b}{\gamma^n} \\\frac{1}{\beta^n} & \frac{d}{\gamma^n} \end{pmatrix}. $$
We begin by establishing the irreducibility properties of $\A_{\beta,\gamma}$.

\begin{lma}\label{irred}
If if $I \subset \N \setminus \{4\}$ intersects $\{n \in \N: n \geq 5\}$ and $\gamma$ is taken sufficiently large, then $\{A^{\beta, \gamma}_i\}_{i \in I}$ is strongly irreducible.
\end{lma}

\begin{proof}
$A^{\beta, \gamma}_1$ has eigenvectors $\left\{\begin{pmatrix}1\\\frac{1}{2}\end{pmatrix}, \begin{pmatrix}  -1\\\frac{1}{2}\end{pmatrix}\right\}$. $\begin{pmatrix}1\\\frac{1}{2}\end{pmatrix}$ cannot be an eigenvector for $A_n^{\beta,\gamma}$ ($n \geq 5$), since this would imply that $\frac{1}{\beta^n}= \frac{d/2-b}{\gamma^n}$, which is impossible since $d/2-b<0$. $\begin{pmatrix}  -1\\\frac{1}{2}\end{pmatrix}$ cannot be an eigenvector of $A_n^{\beta,\gamma}$ ($n \geq 5$) since this would imply that $\frac{3}{\beta^n}=\frac{d/2+b}{\gamma^n}$ which is also impossible provided $\gamma$ is always taken sufficiently larger than $\beta$. Therefore we can always take $\gamma$ sufficiently large that $A_1^{\beta,\gamma}$ does not share any eigenvectors with $A_n^{\beta,\gamma}$ for any $n \geq 5$. Next, we also claim that given $n >m \geq 5$, $A_n^{\beta,\gamma}$ and $A_m^{\beta,\gamma}$ do not share any common eigenvectors. To see this, notice that the positive eigenvector for $A_n^{\beta,\gamma}$ is given by $\begin{pmatrix}1\\v_+\end{pmatrix}$ and the other eigenvector is given by $\begin{pmatrix}1\\v_-\end{pmatrix}$ where $v_+$ and $v_-$ are the positive and negative solutions to the quadratic equation
$$bv^2+v\left(\frac{\gamma^n}{\beta^n} -d\right)-\frac{\gamma^n}{\beta^n}=0$$
which are given by
$$v_\pm= \frac{d-\frac{\gamma^n}{\beta^n} \pm \sqrt{(\frac{\gamma^n}{\beta^n}-d)^2+4b\frac{\gamma^n}{\beta^n}}}{2b}.$$
Hence it is enough to check that $G_\pm'(x) \neq 0$ for $x \geq 1$ where 
$$G_\pm(x):= -f(x) \pm \sqrt{(f(x)-d)^2+4bf(x)}$$
and $f(x):= (\frac{\gamma}{\beta})^x$. Since
$$G_\pm'(x)=f'(x)\left(-1 \pm \frac{f(x)-d}{\sqrt{(f(x)-d)^2+4bf(x)}}\right)$$
and $f(x)>0$ for $x \geq 1$, it is clear that $G_\pm'(x) \neq 0$ for any $x \geq 1$.
Therefore, if $I \subset \N \setminus \{4\}$ intersects $\{n \in \N: n \geq 5\}$ then $\{A^{\beta, \gamma}_i\}_{i \in I}$ is irreducible and therefore strongly irreducible (since any set of positive irreducible matrices is automatically strongly irreducible).
\end{proof}

 Throughout this section we will denote $\I=\{1,2,3\}$ and $\I'=\{1,2,5,6, \ldots\}$. Given $I \subset \N \setminus \{4\}$, we let $s_{\beta, \gamma}(I)$ denote the root of the pressure function
$$P_I^{\beta, \gamma}(s)= \lim_{n \to \infty} \left( \sum_{i \in I^n} \phi^s(A_\i^{\beta, \gamma})\right)^{\frac{1}{n}}.$$ In particular $s_{\beta,\gamma}(\I)=\frac{\log 3}{\log \beta}<1$. Notice that we can choose $c>0$ sufficiently small that for any $\gamma>\beta$ sufficiently large,
\begin{eqnarray}
\min \left\{\frac{(A_k^{\beta, \gamma})_{(i,j)}}{(A_k^{\beta, \gamma})_{(i',j)}} \; : \; k \in \N \setminus \{4\}\right\} \geq c. \label{column}
\end{eqnarray}
We always assume that $\gamma>\beta$ is sufficiently large that: (i) (\ref{column}) holds, (ii) $A_1^{\beta,\gamma}$ is a positive matrix and (iii) for any $(i,j) \in \{1,2\}^2$, $(A_1^{\beta,\gamma})_{(i,j)} \geq (A_5^{\beta,\gamma})_{(i,j)}$ and (iv) Lemma \ref{irred} holds. 

Consider an IFS $\F$ whose linear parts coincide with $\A_{\beta, \gamma}$ and which satisfies the SOSC. Observe that by Theorem \ref{bhr}, for all $I \subset \N \setminus \{4\}$ where $I$ is not contained in $\I$, we have $\hd F_I=s_{\beta, \gamma}(I)$. On the other hand, for $I \in \{\{1,2\}, \{1,3\}, \{2,3\}, \I\}$, we only know that $\hd F_I \leq s_{\beta, \gamma}(I)$ and its exact value depends on the translations in the IFS $\F$. We will exploit this to construct an IFS $\F$ for which $D(\F)$ is not compact .

We begin by studying the structure of $D(\A_{\beta, \gamma})$; in particular we show that for some choice of $\beta$ and $\gamma$: (i) there exists $\epsilon>0$ for which $D(\A_{\beta, \gamma}) \cap (s_{\beta, \gamma}(\I)-\epsilon, s_{\beta, \gamma}(\I))= \emptyset$, (ii) for any $\delta>0$, $D(\A_{\beta, \gamma}) \cap (s_{\beta, \gamma}(\I), s_{\beta, \gamma}(\I)+\delta)\neq \emptyset$ and (iii) for any $I \subset \N \setminus \{4\}$ where $I \neq \I$, we have $s_{\beta, \gamma}(I) \neq s_{\beta, \gamma}(\I)$. Throughout the rest of the section we fix some $\eta \in (0,1)$ and $K=\frac{8}{c \eta}$.

\begin{lma}\label{sI}
For all $\beta>3$ sufficiently large, 
$$\beta^{2s_{\beta,\gamma}(\I)}-\beta^{s_{\beta,\gamma}(\I)}>K^{s_{\beta,\gamma}(\I)}.$$
\end{lma}

\begin{proof}
Since $s_{\beta,\gamma}(\I)= \frac{\log 3}{\log \beta}$, the left hand side of the above inequality equals 6 whereas the right hand side is $K^{\frac{\log 3}{\log \beta}}$, from which the proof immediately follows.
\end{proof} 

From now on we fix $\beta$ sufficiently large such that Lemma \ref{sI} holds.

\begin{lma} \label{crucial}
For all $\gamma$ sufficiently large,
\begin{eqnarray}
s_{\beta, \gamma}(\I')<s_{\beta, \gamma}(\I). \label{crucial eqn} \end{eqnarray}
\end{lma}

\begin{proof}
Denote $s=s_{\beta,\gamma}(\I)$. By Lemma \ref{positive change},
\begin{eqnarray}
P^{\beta, \gamma}_{\I'}(s) &\leq& P^{\beta, \gamma}_{\{1,2\}}(s)+ \sum_{m=5}^{\infty} \frac{4^s (\norm{A^{\beta, \gamma}_m}')^s}{c^s} \nonumber\\
&\leq&  P^{\beta, \gamma}_{\{1,2\}}(s)+ \left(\frac{8}{c}\right)^s \sum_{m=5}^{\infty} \frac{1}{\beta^{ms}}+ \left(\frac{4}{c}\right)^s \sum_{m=5}^{\infty}\frac{(b+d)^s}{\gamma^{ms}} \nonumber\\
&=&P^{\beta, \gamma}_{\{1,2\}}(s)+ \frac{K^s \eta^s}{\beta^{4s}} \frac{1}{\beta^s-1}+\left(\frac{4}{c}\right)^s \sum_{m=5}^{\infty}\frac{(b+d)^s}{\gamma^{ms}} \nonumber \\
&<& P^{\beta, \gamma}_{\{1,2\}}(s)+ \frac{\eta^s}{\beta^{3s}} + \left(\frac{4}{c}\right)^s \sum_{m=5}^{\infty}\frac{(b+d)^s}{\gamma^{ms}}, \label{crucial bound}
\end{eqnarray}
where the final inequality follows because $\beta^{2s}-\beta^s>K^s$. Note that $P^{\beta, \gamma}_{\{1,2\}}(s)+ \frac{\eta^s}{\beta^{3s}}=\frac{2}{\beta^s}+\frac{\eta^s}{\beta^{3s}}<\frac{3}{\beta^s}=1$. As $\gamma \to \infty$ the third term in (\ref{crucial bound}) tends to 0, hence
$$P^{\beta, \gamma}_{\{1,2\}}(s)+ \frac{\eta^s}{\beta^{3s}} + \left(\frac{4}{c}\right)^s \sum_{m=5}^{\infty}\frac{(b+d)^s}{\gamma^{ms}}<1$$
for sufficiently large $\gamma$. In particular we have $s_{\beta, \gamma}(\I')<s_\beta(\I)$.
\end{proof}

The following straightforward result will allow us to compare the pressure of different subsystems.

\begin{prop} \label{compare}
Suppose $n >m$ and $I \subset \N \setminus \{4\}$ with $n, m \notin I$. Then 
$$s_{\beta, \gamma}(I \cup\{n\}) \leq s_{\beta, \gamma}(I \cup \{m\}).$$
\end{prop}

\begin{proof}
For any $\i \in (I \cup \{n\})^k$, notice that we can replace every instance of the digit $n$ with $m$ to obtain a unique word $\j \in (I \cup \{m\})^k$. Therefore, for any $(i,j) \in \{1,2\}^2$, $(A^{\beta, \gamma}_\i)_{(i.j)} \leq (A^{\beta, \gamma}_\j)_{(i,j)}$. In particular $\norm{A^{\beta, \gamma}_\i}' \leq \norm{A^{\beta, \gamma}_\j}'$. Therefore it immediately follows that $P^{\beta, \gamma}_{I \cup\{n\}}(s) \leq P^{\beta, \gamma}_{I \cup \{m\}}(s)$ and hence $s_{\beta, \gamma}(I \cup\{n\}) \leq s_{\beta, \gamma}(I \cup \{m\})$.
\end{proof}

We are now ready to show that $(s_{\beta, \gamma}(\I'), s_{\beta, \gamma}(\I)) \cap D(\A_{\beta, \gamma})=\emptyset$ and that the point $s_{\beta, \gamma}(\I)$ is uniquely attained in the affinity spectrum.

\begin{lma} \label{unique hole}
Let $\gamma$ be sufficiently large that Lemma \ref{crucial} holds. Then $(s_{\beta, \gamma}(\I'), s_{\beta, \gamma}(\I)) \cap D(\A_{\beta, \gamma})=\emptyset$. Moreover, for any $I \subset \N \setminus \{4\}$ where $I \neq \I$ we have $s_{\beta, \gamma}(I) \neq s_{\beta, \gamma}(\I)$.
\end{lma}

\begin{proof}
Fix $I \subset \N \setminus \{4\}$. We consider the following four cases.

\emph{Case 1:} If $\{1,2\} \subset I$ then either (a) $3 \in I$ and therefore $s_{\beta, \gamma}(I) \geq s_{\beta, \gamma}(\I)$, or (b) $3 \notin I$ and then $s_{\beta, \gamma}(I) \leq s_{\beta, \gamma}(\I')$. 

\emph{Case 2:} If $1,2 \notin I$ then $s_{\beta, \gamma}(I)\leq s_{\beta, \gamma}(\{3,5,6,7, \ldots\}) < s_{\beta, \gamma}(\{2,5,6,7,\ldots\})<s_{\beta, \gamma}(\I')$. 

\emph{Case 3:} If $1 \in I$ but $2 \notin I$, then $s_{\beta, \gamma}(I) \leq s_{\beta, \gamma}(\{1,3,5,6,7,\ldots\}) < s_{\beta, \gamma}(\{1,2,5,6,7,\ldots\})=s_{\beta, \gamma}(\I')$.

\emph{Case 4:} If $1 \notin I$ but $2 \in I$ then $s_{\beta, \gamma}(I) \leq s_{\beta, \gamma}(\{2,3,5,6,7,\ldots\}) \leq s_{\beta, \gamma}(\{1,2,5,6,7,\ldots\})=s_{\beta, \gamma}(\I')$.

This proves that $(s_{\beta, \gamma}(\I'), s_{\beta, \gamma}(\I))\cap D(\A_{\beta, \gamma})=\emptyset$. To see that $s_{\beta, \gamma}(\I)$ is uniquely attained in $D(\A_{\beta, \gamma})$, let $I \subset \N \setminus \{4\}$ such that $s_{\beta, \gamma}(I) \geq s_{\beta, \gamma}(\I)$. From examining the above cases we see that necessarily $\{1,2,3\} \subset I$. Therefore, either $I=\I$, or there must exist $n \in \N \setminus \{1,2,3,4\}$ such that $n \in I$. In the latter case, we know by Lemma \ref{positive change} that $P^{\beta, \gamma}_I(s)>P^{\beta, \gamma}_\I(s)$ for any $s \in (0,1)$ hence $s_{\beta, \gamma}(I)>s_{\beta, \gamma}(\I)$. 
\end{proof}

We are now ready to prove Theorem \ref{main}.

\vspace{2mm}

\noindent \emph{Proof of Theorem \ref{main}.} By Lemma \ref{unique hole} and Theorem \ref{bhr}, it is sufficient to show that we can find translational parts $u_1, u_2, u_3 \in \R^2$ such that $\{A_i^{\beta,\gamma} (\cdot) +u_i\}_{i \in \I}$ satisfies the SOSC and its attractor $F_\I$ has dimension $\hd F_\I<s_{\beta, \gamma}(\I)= \frac{\log 3}{\log \beta}$. By conjugating, it is sufficient to show that we can choose translations $t_1, t_2, t_3 \in \R^2$ in such a way that the IFS $\{A^{\beta,\gamma}(\cdot)+t_i\}_{i \in \I}$ satisfies the SOSC and its attractor $F_\I'$ has dimension $\hd F_\I'< s_{\beta,\gamma}(\I)=\frac{\log 3}{\log \beta}$. Now, let $t_i= \begin{pmatrix} 0\\ \frac{i-1}{\gamma}\end{pmatrix}$ for $i=1,2,3$. It is easy to see that the SOSC is satisfied. Moreoever, $F_\I'$ is just a rotated copy of the attractor of the self-similar IFS $\{\frac{1}{\gamma}x+ \frac{i-1}{\gamma}\}_{i \in \I}$, which has dimension $\frac{\log 3}{\log \gamma}< \frac{\log 3}{\log \beta}=s_{\beta,\gamma}(\I)$. Therefore we are done. \qed

\begin{rem} With some more work, one can show that it is possible to choose $\beta, \gamma$ and the translational parts in such a way that $\hd F_\I \in (s_{\beta, \gamma}(\I'), s_{\beta,\gamma}(\I))$, showing that $\hd F_\I $ is an isolated point in the dimension spectrum of an appropriate self-affine IFS $\mathcal{F}$. We do not pursue this however, and instead prove Theorem \ref{isol} using a much simpler construction in the next section.
\end{rem}

\subsection{Example of a dimension spectrum containing an isolated point} \label{annoying}

In this section we construct a self-affine IFS whose dimension spectrum contains an isolated point, thus proving Theorem \ref{main}. Note that the example studied in this section falls into a class of IFS studied by Reeve in \cite{reeve}, although we will not require any of his results for our analysis.

Let 
\begin{eqnarray*}
S_1 \begin{pmatrix} x\\y \end{pmatrix} &=&\begin{pmatrix} \frac{1}{3} &0 \\0& \frac{1}{4} \end{pmatrix} \begin{pmatrix} x\\y \end{pmatrix}\\
S_2 \begin{pmatrix} x\\y \end{pmatrix} &=&\begin{pmatrix} \frac{1}{3} &0 \\0& \frac{1}{4} \end{pmatrix} \begin{pmatrix} x\\y \end{pmatrix}+ \begin{pmatrix} 0\\ \frac{1}{2}\end{pmatrix} \\
\end{eqnarray*}
and for $n \geq 3$,
\begin{eqnarray*}
S_n \begin{pmatrix} x\\ y\end{pmatrix}&=& \begin{pmatrix} \frac{1}{3} &0 \\0& a_n \end{pmatrix} \begin{pmatrix} x\\y \end{pmatrix}+ \begin{pmatrix} \frac{2}{3}\\ t_n\end{pmatrix}
\end{eqnarray*}
 where $a_n \in (0,\frac{1}{3})$ are decreasing in $n$ and the vertical component $t_n$ for the translational part of $S_n$ can be chosen arbitrarily provided that each $S_n([0,1]^2) \subseteq [0,1]^2$ and $S_n([0,1]^2) \cap S_m([0,1]^2) =\emptyset$ for $n \neq m$. Additionally we assume that the entries $a_n$ are sufficiently small so that the solution $s_0$ to the equation $\sum_{n=3}^\infty a_n^{s_0}=1$ satisfies $ s_0<\frac{\log 2}{\log 4}$.

\vspace{2mm}

\noindent \textit{Proof of Theorem \ref{isol}.} 
The attractor of $\{S_1, S_2\}$  is just the self-similar attractor of $\{\frac{1}{4}x, \frac{1}{4}x+\frac{1}{2}\}$ rotated anticlockwise by 90 degrees, so that it lies on the $y$-axis. Therefore $\hd F_{\{1,2\}}=\frac{\log 2}{\log 4}$. If $\I \subset \{3,4,5, \ldots\}$ then the attractor of $\{S_i\}_{i \in \I}$ is just a rotated and translated copy of the self-similar attractor of $\{a_ix+t_i\}_{i \in \I}$, hence $\hd F_\I \leq s_0<\frac{\log 2}{\log 4}$. Finally, if $\I$ intersects both $\{3,4,5,\ldots\}$ and $\{1,2\}$, then the dimension of the attractor of $\{S_i\}_{i \in \I}$ is bounded below by the dimension of its projection to the $x$-axis, which is simply the middle-third Cantor set, hence $\hd F_\I \geq \frac{\log 2}{\log 3}$. It follows that $\hd F_{\{1,2\}}$ is an isolated point in the dimension spectrum. \qed

\begin{rem}
\emph{Finite} subsystems of $\{S_n\}_{n \in \N}$ fall into the class of self-affine IFS that were studied by Lalley and Gatzouras in \cite{lg}, who obtained an explicit formula for the Hausdorff dimension in \cite[Proposition 3.3]{lg} in terms of a variational principle. To compute the Hausdorff dimension of infinite subsystems of $\{S_n\}_{n \in \N}$, this formula can be combined with \cite[Theorem 1]{reeve}. Examining the formula for the Hausdorff dimension of subsystems of $\{S_n\}_{n \in \N}$, there do not appear to be any obvious violations to compactness of the dimension spectrum, although the nature of the formula (in that it is expressed as a maximum over dimensions of projected Bernoulli probability measures) makes this difficult to conclude without further analysis.
\end{rem}

\noindent \textbf{Acknowledgements.}  The author was financially supported by the \emph{Leverhulme Trust} (Research Project Grant number RPG-2016-194) and by the \emph{EPSRC} (Standard Grant EP/R015104/1).

\end{document}